\documentclass[12pt]{amsart}
\usepackage{latexsym,amssymb,amsmath}
\usepackage[dvips]{graphicx}
\usepackage{xcolor}

\usepackage{graphics,xy,epsfig,picture,epic}

\def\ff{{\mathcal F}}
\def\hh{{\mathcal H}}

\def\lll{{\mathcal L}}

\def\ffi{\varphi}
\def\eps{\varepsilon}
\def\dst{\displaystyle}

\renewcommand{\Im}{\mathrm{Im}\,}

\def\C{{\mathbb{C}}}

\def\R{{\mathbb{R}}}

\def\T{{\mathbb{T}}}

\def\Z{{\mathbb{Z}}}

\def\tree{{\mathbf{T}}}

\def\d{\,{\mathrm{d}}}
\newcommand{\norm}[1]{{\left\|{#1}\right\|}}
\newcommand{\ent}[1]{{\left[{#1}\right]}}
\newcommand{\abs}[1]{{\left|{#1}\right|}}

\newtheorem{lemma}{Lemma}[section]
\newtheorem{proposition}[lemma]{Proposition}
\newtheorem{theorem}[lemma]{Theorem}
\newtheorem{corollary}[lemma]{Corollary}
\newtheorem{theoreml}{Theorem}
\newtheorem{corl}[theoreml]{Corollary}

\theoremstyle{definition}
\newtheorem{definition}[lemma]{Definition}

\theoremstyle{remark}
\newtheorem{remark}[lemma]{Remark}

\begin{document}

\title[Schr\"odinger equation on trees]
{ Uniqueness for solutions of the Schr\"odinger equation on trees}

\begin{abstract}
We prove that if a solution of the time-dependent Schr\"oding\-er equation on an homogeneous tree with bounded potential decays fast at 
two distinct times then the solution is trivial. For the free Schr\"odinger operator, we use the spectral theory of the Laplacian
and complex analysis and obtain a characterization of the initial conditions that lead to a sharp decay at any time.
We then adapt the real variable methods first introduced by Escauriaza, Kenig, Ponce and Vega to establish a general sharp result in the case of bounded potentials.
\end{abstract}

\author{Aingeru Fern\'{a}ndez-Bertolin}
\address{(A. Fern\'andez) UPV/EHU, Dpto. Matem\'aticas, Barrio Sarriena s/n 48940 Leioa, Spain.}
\email{aingeru.fernandez@ehu.eus}

\author{Philippe Jaming}
\address{(Ph. Jaming) Univ. Bordeaux, CNRS, Bordeaux INP, IMB, UMR 5251,  F-33400, Talence, France.}
\email{Philippe.Jaming@math.u-bordeaux.fr}



\maketitle

\section{Introduction}

The aim of the present paper is to study uniqueness results for Schr\"odinger equations with bounded potentials on homogeneous trees. 
These results can be seen as a version for homogeneous trees of a dynamical interpretation of the Hardy Uncertainty Principle.

The Schr\"odinger equation $i\partial_t u=\Delta u +Vu$ has been extensively studied by mathematicians and physicists.
Those studies take place in various underlying spaces, both continuous ($\R^d$, manifolds,...) and discrete. In the discrete setting,
on $\Z^d$, and on the homogeneous tree (also known as a {\em Bethe lattice} in the physics community) this equation has first
been considered by Anderson and collaborators
in \cite{An,ACTA} in order to describe the behavior of a quantum particle in disordered medium.

Our aim here is to further investigate properties of solutions of the Schr\"odinger equation on an homogeneous tree.
We will prove that solutions of the Schr\"odinger equation $i\partial_t u=\Delta u +Vu$
on an homogeneous tree can not be too sharply localized at 2 different times when the potential $V$ is bounded.
Our results may thus be seen as a dynamical version of the Uncertainty Principle.
Before outlining our results more precisely, let us first explain what we mean by ``localizing'' and further
explain our motivations in this paper.

Let us start by recalling Hardy's uncertainty principle \cite{ha} on the real line: assume $f\in L^2(\R)$ satisfies
a decrease property like
\begin{equation}
\label{eq:pointwiseHardy}
|f(x)|\le C e^{-x^2/\beta^2},\ |\hat{f}(\xi)|\le C e^{-4\xi^2/\alpha^2}.
\end{equation}
Then, if $\alpha\beta<4$, $f\equiv0$ while, in the end-point case, $\dst\frac{1}{\alpha\beta}=\frac14,\ f=C e^{-x^2/\beta^2}$.
In other words, a function and its Fourier transform can not both be localized below two sharply localized Gaussians.

Numerous authors have extended this result to higher dimensions, replacing the point-wise estimate \eqref{eq:pointwiseHardy}
by integral or even distributional conditions ({\it see e.g.} the works of H\"ormander, Bonami, Demange and the second author
\cite{Ho,bdj,BD,De})
and also replacing the underlying space $\R^d$ by various Lie groups (as can be found for instance in
the work of Baklouti, Kaniuth, Sitaram, Sundari, Thangavelu,... including \cite{BK1,BK2,sst,Th1,Th2}). The survey \cite{FS} and the books \cite{HJ,Th3} may be taken as a starting point to further investigate the subject.
Most of this work requires either complex analysis or a reduction to a real variable setting in which complex variable tools are available.
A first difficulty appears here as the decrease in the space variable and in the Fourier variable can no longer be measured in the same
way. This problem becomes even more striking in the discrete setting. For instance, for functions on $\Z$, the Fourier transform
is a periodic function, so that there is no decrease at infinity.

To overcome this, one way is to consider a dynamical interpretation of the uncertainty principle.
To explain what we mean by this, let us go back to the real line.
Recall that the solution of the free Schr\"odinger equation $i\partial_tu=\Delta u$, $u(0,x)=u_0(x)$
is given by the following representation formula:
\begin{multline*}
u(x,t)=(4\pi i t)^{-n/2}\int_{\mathbb{R}^n}e^{\frac{-i|x-y|^2}{4t}}u_0(y)\,dy\\
=(2\pi it)^{-n/2}e^{\frac{-i|x|^2}{4t}}\widehat{e^{-i\frac{|\cdot|^2}{4t}}u_0}\left(-\frac{x}{2t}\right).
\end{multline*}
Hence, the solution at a fixed time has, roughly speaking, the same size as the Fourier transform of the initial data, and we can translate decay properties of $u_0$ and $\widehat{u_0}$ into decay properties of $u_0$ and $u(x,T)$ for a fixed time $T$, to have
$$
|u_0(x)|\le C e^{-x^2/\beta^2},\ |u(x,T)|\le C e^{-x^2/\alpha^2},\ \frac{T}{\alpha\beta}>\frac14\Longrightarrow u\equiv 0
$$
and, if $\dst\frac{T}{\alpha\beta}=\frac14,\ u_0(x)=C e^{-x^2(1/\beta^2+i/4T)}$.

This point of view has been used by Chanillo \cite{Ch} to prove a dynamical uncertainty principle
on complex semi-simple Lie groups by reducing the problem to Hardy's Uncertainty Principle on the real line.
At the same time, Escauriaza, Kenig, Ponce and Vega started a series of papers \cite{ekpv1,ekpv2,ekpv3}
were they provide the first proof of Hardy's Uncertainty Principle in its dynamical version in the presence of a potential,
using real calculus. 
Their motivation is to consider solutions of general linear Schr\"odinger equations $i\partial_tu=\Delta u +Vu$,
only assuming size conditions for the space and time-dependent potential $V$. The robustness of both methods allows to extend their results
to different settings, such as for covariant Schr\"odinger evolutions by
Barcel\'o, Cassano, Fanelli, Guti\'errez, Ruiz, Vilela \cite{bfgrv,cf}, or heat evolutions \cite{ekpv4}
but also to other underlying spaces, {\it see e.g.} the work of Ben Sa\"id, Dogga, Ludwig, M\"uller, Pasquale, Sundari, Thangavellu \cite{BSTD,LuMu,PS}.

More recently, independently in \cite{fb,fbv,jlmp}, together with Lyubarskii, Malinnikova, Perfekt and Vega, we began to extend the previous results to the discrete setting, 
understanding the Laplace operator as a finite-difference operator, acting on complex-valued functions
$f:\mathbb{Z}\rightarrow \mathbb{C}$,
$$
\Delta_d f(n):=f(n+1)+f(n-1)-2f(n).
$$

For the free evolution, or for the linear evolution with a bounded time-independent potential, as shown in \cite{lm}, 
one can use complex analysis tools, more precisely refined versions of the Phragm\'en-Lindel\"of principle,
to give a discrete version of the Hardy Uncertainty Principle. As in the continuous case,
the critical decay is given by the discrete heat kernel, given in terms of modified Bessel functions. However, this similarity
leads also to the main difference between both settings, because the critical decay is not Gaussian. More precisely,
it is shown in \cite{jlmp} that for $0<\alpha<1$ and $u$ a $C^{1}([0,1],\ell^2(\mathbb{Z}))$-solution of $\partial_tu=i\Delta_du$ (a so-called \textit{strong} solution), if $u$ satisfies the estimate
\begin{equation}
\label{eq:besselest}
|u(n,0)|+|u(n,1)|\le C I_n(\alpha) \sim \frac{C}{\sqrt{|n|}}\left(\frac{e\alpha}{2|n|}\right)^{|n|},\ n\in\mathbb{Z}\setminus\{0\},
\end{equation}
then $u\equiv0$. In the end-point case, $\alpha=1$, $u(n,t)=\gamma i^{-n}e^{-2it}J_n(1-2t)$, where $\gamma$
is a constant and $J_n$ is the Bessel function.
Note that classical estimates of Bessel functions show that, for any $\gamma>0$, there is a $C>0$ such that this solution
indeed satifies \eqref{eq:besselest}. This argument is also extended to other type of problems, as shown by Alvarez-Romero
and Teschl \cite{ar,art} for Jacobi operators.

In the case of linear Schr\"odinger equations, one can give a dynamical version of the Hardy Uncertainty Principle, only assuming that 
the potential is bounded, which makes another difference with the continuous case, since in the continuous case,
all results in \cite{ekpv1,ekpv2,ekpv3} require to have some 
decay in the potential, and the result is still open for bounded potentials.
To be more precise, the first author and Vega \cite{fbv} showed that if $u$ is a strong solution of 
$\partial_tu=i(\Delta_d u +Vu)$ on $\Z$ (with $V=V(n,t)$ bounded) and if $u$ satisfies the decay condition
\begin{equation}
\label{eq:decayZ}
\sum_{n\in\mathbb{Z}}e^{2\mu(|n|+1)\log(|n|+1)}(|u(n,0)|^2+|u(n,1)|^2)<\infty
\end{equation}
for some $\mu>1$, then $u=0$. In view of the free case, as $u(n,t)=\gamma i^{-n}e^{-2it}J_n(1-2t)$ is a solution of the free Schr\"odinger equation and it satisfies \eqref{eq:decayZ} with $\mu=1+\epsilon,\ \forall \epsilon>0$ (as one can deduce from \eqref{eq:besselest}), the condition $\mu>1$ is optimal. It is worth to mention that $\mu=1$ gives the leading term in the asymptotic expression for $I_n(\alpha)$ in \eqref{eq:besselest}.
Note also
that \cite{jlmp,fb} both contain similar results but only in non-optimal cases $\mu>\mu_0>1$.
A higher dimensional version of this result can be found in \cite{fbv}, although the rate of 
decay $\mu$ obtained there depends on the dimension and the sharp result is still open.
The key tool here is to establish Carleman type estimates, that is, a weighted inequality of the form
$C_w\norm{w u}_{L^2(\Z^d)}\leq \norm{w(i \partial_t +\Delta_d)u}_{L^2(\Z^d)}$ for an appropriate weight $w$
and a constant $C_w$ depending on this weight. We refer to \cite{LR} for more on Carleman estimates and their use in the continuous setting.

Therefore the results in \cite{fbv,jlmp,lm} are based on two different approaches. For the linear evolution with time-independent bounded potential one uses complex analysis, while in the presence of a time-dependent bounded potential the Phragm\'en-Lindel\"of principle is not available and one replace this by a suitable Carleman inequality (using real variable methods instead of complex analysis).

\smallskip

In this paper we extend both approaches to homogeneous trees of degree $q+1$ ({\em Bethe lattices}), which we denote by $\tree_q$. 
This is a connected graph with no loops, rooted in a point denoted by $o$, where every vertex is adjacent to $q+1$ other vertices, 
a relation denoted by $y\sim x$.
Thus, one can see $\tree_q$ as a natural extension of the line $\mathbb{Z}$, which can be seen as a homogeneous tree of degree $2$.
One may then ask whether the behavior for solutions of Schr\"odinger evolutions is similar on $\Z$ and on $\tree_q$. 
As in the line $\Z$, we understand the 
Laplacian as the combinatorial Laplacian, that is
a finite-difference operator $\mathcal{L}$ only taking into account interactions between nearest-neighbors
(see Section 2 for a precise definition). 

It is our aim here to contribute to the understanding of the behavior
of solutions of Schr\"odinger equations on trees ({\it see e.g.} the recent papers by Anantharaman, Colin de Verdi\`ere, Eddine,
Sabri, Truc \cite{AS,Ed,CdVT} for other directions)
by establishing Uncertainty Principles on trees (so far, we are only aware of one article by Astengo \cite{Ast} dealing with that
issue).

We are now in position to describe our results. First, since the spectral theory of the Laplacians on homogeneous trees is known 
({\it see} Cowling and Setti \cite{cs}), we have all the ingredients to give a dynamic interpretation of the Hardy Uncertainty Principle on $\tree_q$
when there is no potential:

\begin{theoreml}
\label{thA}
There exists a function $U_q$ on $\tree_q$ such that, if $u$ is a strong solution of the equation 
$$
i\partial_tu(x,t)=\mathcal{L}u(x,t)=u(x,t)-\frac{1}{q+1}\sum_{y\sim x}u(y,t),\ \ x\in\tree_q
$$
with $u(x,0)=u_0(x)$ and if at times $t_0=0$ and $t_1=1$, there is a $\kappa$ such that, for $x\not=o$
\begin{equation}
\label{eq:thmA}
|u(x,t_i)|\le \frac{\kappa}{\sqrt{|x|}}\left(\frac{e}{2(q+1)|x|}\right)^{|x|}
\end{equation}
then $u_0=\gamma U_q$ for some $\gamma\in\C$.
\end{theoreml}

The function $U_q$ is explicitly given by an integral formula, {\it see} below.
In order to compare our results with the case of $\Z$, let us rewrite \eqref{eq:thmA} as
$$
|u(x,t_i)|\le \kappa |x|^{-1/2} e^{\bigl(1-\ln 2(q+1)\bigr)|x|} e^{-|x|\ln |x|}.
$$
We thus see that the main term $e^{-|x|\ln |x|}$ does not depend on the tree and is the same as for $\Z$ and
that the dependence on the degree of the tree is rather mild. It is somewhat unexpected that the behavior is the same
in both cases as the tree is the Caley-graph of the free group which is non-amenable and has exponential growth while
$\Z$ is amenable and has polynomial growth.

Further, as an immediate corollary, we obtain
\begin{corl}
\label{corB}
Let $\mu>1$. If $u$ is a strong solution of the equation 
$$
i\partial_tu(x,t)=\mathcal{L}u(x,t)=u(x,t)-\frac{1}{q+1}\sum_{y\sim x}u(y,t),\ \ x\in\tree_q
$$
with $u(x,0)=u_0(x)$ and if at times $t_0=0$ and $t_1=1$, 
\begin{equation}
\label{eq:corB}
\sum_{x\in\tree_q}e^{2\mu|x|\log(|x|+1)}\big(|u(x,0)|^2+|u(x,1)|^2\big)<+\infty
\end{equation}
then $u\equiv0$.
\end{corl}
Our second aim is to show that this corollary stays true for the Schr\"odinger equation
in presence of a potential: $i\partial_tu(x)=\mathcal{L}u(x)+V(x,t)u(x)$, with a bounded time-dependent potential $V$.
This time, we will use real variable calculus. 

This approach combines the main techniques of \cite{fbv,jlmp}, to prove first that a fast decaying solution at two different times preserves this decay at any interior time, and, later, via a Carleman estimate with Gaussian weight, we give a lower bound for the $\ell^2-$norm of the solution in a region far from the origin ({\it see} Theorem \ref{th:4.7} below). A combination of these two facts leads then to:

\begin{theoreml}[Uniqueness result]
\label{thC}
Let $u\in C^1([0,1]:\ell^2(\tree_q))$ be a solution of $i\partial_tu(x)=\mathcal{L}u(x)+V(x,t)u(x)$ with $V$ a bounded potential. If for $\mu>1$
\[
\sum_{x\in\tree_q}e^{2\mu|x|\log(|x|+1)}\big(|u(x,0)|^2+|u(x,1)|^2\big)<+\infty,
\]
then $u\equiv0$.
\end{theoreml}

This shows that Corollary \ref{corB} is also valid in the presence of a bounded potential, in particular,
the condition $\mu>1$ is essentially sharp up to the end-point $\mu=1$ which is open.
This result is exactly the same as in the
case of $\Z$, \cite{fbv}. This is no longer surprising in view of Theorem \ref{thA} and Corollary \ref{corB} as the influence
of the tree in the optimal decay is very mild. However, one may ask if this result is true for any infinite graph,
or if it can be extended to large classes of graphs. We provide some examples of infinite graphs for which
the behavior of the solutions is different.


\medskip

The paper is organized as follows: in Section 2 we introduce some notation and preliminaries from the theory of entire functions as well as a summary of the spectral theory of the adjacency matrix on $\tree_q$. These notions can be found in \cite{CdVT,Lbook}, but we include them here to clarify our presentation. Section 3 studies the free Schr\"odinger equation and includes the proof of Theorem \ref{thA}. 
Section \ref{carleman} covers the real variable approach, proving Theorem \ref{thC} via a Carleman inequality and logarithmic convexity of $\ell^2$ weighted norms. We conclude in Section \ref{other} with some considerations on other graphs.

\section{Notation and preliminaries}

\subsection{Entire functions of exponential type}

As in \cite{jlmp}, we will use methods from complex analysis. For the reader's convenience,
we begin by briefly outlining some definitions and facts on entire functions of exponential
type that we need. Details can be found in \cite{Lbook} (see in particular Lectures 8 and 9).
Recall that an entire function $f$ is said to be of exponential type if for some $k>0$
\begin{equation}
\label{eq:exptype}
|f(z)|\le C\exp(k|z|).
\end{equation}  
In this case the type of an entire function $f$ is defined by 
\begin{equation}
\sigma =\limsup_{r\to\infty}\frac{\log \max\{|f(re^{i\phi})|;\phi\in[0,2\pi]\}}{r} <\infty.
\end{equation}
In particular, an entire function $f$ is of zero exponential type if for any $k>0$ there exists $C=C(k)$ such that 
\eqref{eq:exptype} holds.

Let $f(z)$ be an entire function of exponential type, $f(z)=\sum_{n=0}^\infty c_nz^n$. Then the type of $f$ can be  expressed in terms of its Taylor coefficients as
\begin{equation}\label{eq:coeftype}
 \limsup_{n\to\infty} n|c_n|^{1/n} = e\sigma. 
 \end{equation}

The growth of a function $f$ of exponential type along different directions is described  by the indicator function
\[
h_f(\varphi)=\limsup_{r\to\infty}\frac{\log|f(re^{i\varphi})|}{r}.
\]
This function is the support function of some convex compact set $I_f\subset \C$ which is called the indicator diagram of $f$\,:
\[
h_f(\varphi)=\sup\{\Re(ae^{-i\varphi}),a\in I_f\}.
\]
 In particular
\begin{equation}
\label{eq:if}
h_f(\varphi)+h_f(\pi+\varphi)\ge 0.
\end{equation} 
For example the indicator function of $e^{az}$ for $a\in \C$ is $h(\varphi)=\Re(ae^{i\varphi})$ and its indicator diagram consists of 
a single  point, $\bar{a}$. 

Clearly, $h_{fg}(\varphi)\le h_f(\varphi)+h_g(\varphi)$, implying that
\[
I_{fg}\subset I_f+I_g:=\{z=z_1+z_2: z_1\in I_f, z_2\in I_g\}.
\] 

\subsection{Trees}

In this section, we recall some basics of harmonic analysis on trees. For more on this subject,
one may refer to {\it e.g.} \cite{cs,CSM,FTN,FTP} and references therein.

Throughout this paper, $q$ will be an integer, $q\geq 2$. We will denote by $\tree=\tree_q$ the homogeneous
tree of degree $q+1$. This means that the tree is formed by a connected graph with no loops where every vertex is adjacent 
to $q+1$ other vertices, relation denoted by $y\sim x$. 

A {\em geodesic path} (resp. {\em geodesic ray}, {\em infinite geodesic}) in $\tree$ is a finite (resp. one-sided infinite,
resp. doubly infinite) sequence $(x_n)$ such that two consecutive terms are adjacent, $x_n\sim x_{n-1}$ and
that does not turn back $x_{n+1}\not=x_{n-1}$. 
We can then define the distance $\d(x,y)$ as the number of points in the geodesic path which goes from $x$ to $y$. 
In particular, in a geodesic, $\d(x_n,x_m)=|n-m|$.

Moreover, we fix a vertex of the tree $\tree$ to be the root $o$ and write $|x|=\d(x,o)$.
For an integer $\ell\geq 0$, we denote by $S_\ell=\{x\in\tree\,: |x|=\ell\}$.
The boundary $\partial\tree$ of $\tree$ is defined as the set of infinite paths starting at the root $o$. 
Then, we define, for a point $x\in\tree$ and $w\in\partial\tree$, the confluence point of $x$ and $w$, 
denoted by $x\wedge w$ as the last point lying on $w$ in the geodesic path joining $o$ and $x$. Attached to this confluence point we define 
the {\em Busemann function} $h_w$ and the {\em Horocycles} $\hh_k^w$, $k\in\Z$ by
\[
h_w(x)=|x|-2|x\wedge w|\quad,\mbox{and}\quad \hh_k^w=\{x\in\tree\,: h_w(x)=k\}.
\]
We call $k$ the height of the horocycle $\hh_k^w$.
Every horocycle is infinite and every $x\in \hh_k^w$ has one neighbor $x^-\in \hh_{k-1}^w$ (its predecessor) and $q$
neighbors in $\hh_{k+1}^w$ (its successors).

Now let $\psi_{\ell,k}=|S_\ell\cap\hh_k^w|$ be the number of elements in an horocycle $\hh_k$ that are of length $\ell$.
When $k\geq 0$,
$$
\psi_{\ell,k}=\begin{cases}
q^k&\mbox{if }\ell=k\\
(q-1)q^{k+p-1}&\mbox{if }\ell=k+2p,\ p\geq1\\
0&\mbox{otherwise}
\end{cases}
$$
and for $k\geq1$,
$$
\psi_{\ell,-k}=\begin{cases}
1&\mbox{if }\ell=k\\
(q-1)q^{p-1}&\mbox{if }\ell=k+2p,\ p\geq1\\
0&\mbox{otherwise}
\end{cases}.
$$

\begin{figure}[ht!]
\setlength{\unitlength}{0.02cm}
\begin{center}
\begin{picture}(500,300)
\drawline(75,25)(345,295)
\put(335,285){\vector(1,1){10}}
\drawline(85,35)(95,25)

\put(350,290){$\scriptstyle w$}

\put(20,20){$\scriptstyle\hh_{3}$}
\put(20,32){$\scriptstyle\hh_{2}$}
\put(20,47){$\scriptstyle\hh_{1}$}
\put(20,77){$\scriptstyle\hh_{0}$}
\put(20,137){$\scriptstyle\hh_{-1}$}
\put(20,257){$\scriptstyle\hh_{-2}$}

\dottedline{3}(40,25)(450,25)
\dottedline{3}(40,35)(450,35)
\dottedline{3}(40,50)(450,50)
\dottedline{3}(40,80)(450,80)
\dottedline{3}(45,140)(450,140)
\dottedline{3}(45,260)(450,260)

\dottedline{3}(60,1)(450,1)
\put(0,2){$\scriptstyle \partial\tree\setminus\{w\}$}

\dottedline{4}(100,4)(100,17)
\dottedline{4}(200,4)(200,17)
\dottedline{4}(300,4)(300,17)
\dottedline{4}(400,4)(400,17)

\drawline(100,50)(125,25)
\drawline(115,35)(105,25)

\drawline(160,50)(135,25)
\drawline(145,35)(155,25)
\drawline(130,80)(185,25)
\drawline(175,35)(165,25)

\put(130,80){\circle*{6}}
\put(122,88){$o$}

\drawline(250,80)(195,25)
\drawline(205,35)(215,25)
\drawline(220,50)(245,25)
\drawline(235,35)(225,25)

\drawline(190,140)(305,25)
\drawline(295,35)(285,25)
\drawline(280,50)(255,25)
\drawline(265,35)(275,25)

\drawline(340,50)(365,25)

\drawline(355,35)(345,25)
\drawline(430,140)(315,25)
\drawline(325,35)(335,25)
\drawline(370,80)(425,25)
\drawline(415,35)(405,25)
\drawline(385,35)(395,25)
\drawline(400,50)(375,25)

\drawline(440,130)(310,260)



\end{picture}
\caption{The tree $\tree_2$ and horocycles.}
\label{fig:levelset}
\end{center}
\end{figure}
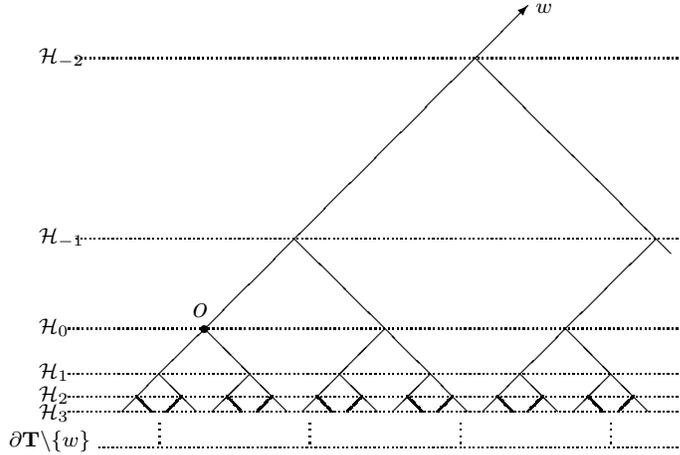

The so called Helgason-Fourier transform ({\it see e.g.} \cite{cs}) of a function $f$ on the tree is defined by the formula
\[
\ff_\tree[f](s,w)
:=\sum_{x\in\tree}f(x)q^{-(1/2+is)h_w(x)},\ \ s\in\T,\ w\in\partial\tree,
\]
where $\T=\R/\tau\Z$, usually identified with the interval $[-\tau/2,\tau/2)$, with $\tau=2\pi/\log q$.

Moreover, the following inversion formula holds,
\begin{equation}
\label{inv}
f(x)=\int_{\T}\int_{\partial\tree} q^{-(1/2-is)h_w(x)}\ff_\tree[f](s,w)\d\nu(w)\d\mu(s),\ x\in\tree.
\end{equation}

We refer to \cite{cs} for the exact definitions of the measures $\nu$ and $\mu$.

We will also need to distinguish between the neighbors and double neighbors of a vertex of the tree. More precisely,
for $x\in\tree$ with $|x|=n$ we set

--- $x_f=\{y\in\tree\,:\ |y|=n+1\}$ and, if $x\not=o$, $x_p$ to be the unique $y\in\tree$ such that $y\sim x$ and $|y|=n-1$.
Note that $|o_f|=q+1$ and, if $x\not=o$, $|x_f|=q$.

--- $x_{ff}=\{y\in\tree\,:\ |y|=n+2,\ y_p\in x_f\}$ so that $|o_{ff}|=q(q+1)$ and, if $x\not=o$, $|x_{ff}|=q^2$.

--- If $|x|\geq 2$, $x_{pp}=(x_p)_p$.

--- If $x\not=o$, $x_{pf}=(x_p)_f\setminus\{x\}$ so that $|y|=|x|$ if $y\in x_{pf}$. Note that if $|x|=1$, $|x_{pf}|=q$ while otherwise $|x_{pf}|=q-1$.

In other words, $x_f$ is the set of followers (daughters) of $x$, $x_p$ the predecessor (mother) of $x$, $x_{pp}$ is the grand-mother of $x$,
$x_{ff}$ the set of grand-daughters of $x$, $x_{pf}$ the set of sisters of $x$.

Note that, for any function $\ffi$ on $\tree$, and any $n\geq 1$,
\begin{equation}
\label{eq:usefull}
\sum_{|x|=n}\sum_{z\in x_{pf}}\ffi(z)=\begin{cases} q\sum_{|x|=1}\ffi(x)&\mbox{if }n=1\\
(q-1)\sum_{|x|=n}\ffi(x)&\mbox{if }n\geq 2
\end{cases}.
\end{equation}

Finally, we consider the adjacency operator $A_0$ and the Laplace operator $\lll$ on $\tree$: for $u$ a function on $\tree$,
$$
A_0u(x)=\sum_{y\sim x}u(y)
$$
and\footnote{Results in this paper can be adapted in a straighforward way to the Laplacian more commonly used in the physics community
$$
\Delta\ffi(x)=\deg x\ffi(x)-\sum_{y\sim x}\ffi(y).
$$}
\begin{eqnarray*}
\lll u(x)&=&\left(I-\frac{1}{q+1}A_0\right) u(x)=u(x)-\frac{1}{q+1}\sum_{y\sim x}u(y)\\
&=&\frac{1}{q+1}\sum_{y\sim x}\bigl(u(x)-u(y)\bigr).
\end{eqnarray*}

We will denote by $\|\cdot\|_2$ the $\ell^2(\tree)$-norm: if $u\,:\tree\to\C$,
$$
\norm{u}^2_2=\sum_{x\in\tree}|u(x)|^2
$$
and by $\|\cdot\|_{L^2_{x,t}}$ the $L_t^2\ell^2-$norm: if $u\,:[0,1]\times\tree\to\C$,
$$
\|u\|_{L^2_{x,t}}=\int_0^1\sum_{x\in\tree}|u(t,x)|^2\,\mbox{d}t.
$$

\section{Free Schr\"odinger equation on the tree}

We want to study uniqueness properties of solutions of the Schr\"o\-din\-ger equation
$i\partial_tu=\mathcal{L}u$
assuming that they have fast decay at two different times.
Adapting the method developed in \cite{jlmp} in the case of the line $\Z$ to the tree, our main result in this section is Theorem A from the introduction, in a slightly more precise form:

\begin{theorem}\ \\
Assume that $u$ is a strong solution of the equation 
\begin{equation}\label{sch}
i\partial_tu(x,t)=\mathcal{L}u(x,t)=u(x,t)-\frac{1}{q+1}\sum_{y\sim x}u(y,t),\ \ x\in\tree
\end{equation}
with $u(x,0)=u_0(x)$. Assume that there is a $\kappa$ such that, at times $t_0=0$ and $t_1=1$, for $x\not=o$
\begin{equation}
\label{hyp}
|u(x,t_i)|\le \frac{\kappa}{\sqrt{|x|}}\left(\frac{e}{2(q+1)|x|}\right)^{|x|}.
\end{equation}

Then there exists a constant $C$ such that
$u_0$ is the function that only depends on $|x|$ given by the integral representation formula
\[
u_0(x)=\frac{C}{q^{|x|/2}}\int_{0}^\pi \exp\left(-i\frac{q^{1/2}}{q+1}\cos(z)\right)
\ffi_{|x|}(z)\sin(z)\d z
\]
where 
$$
\ffi_j(z)=\frac{q^{1/2}\sin\bigl(z(j+1)\bigr)-q^{-1/2}\sin\bigl(z(j-1)\bigr)}{q+q^{-1}-2\cos(2z)}.
$$
\end{theorem}

\begin{remark} A change of variable allows us to write $u_0$ as
$$
u(|x|,0)=\frac{C}{q^{|x|/2}}\ff[\psi_{|x|}]\left(\frac{q^{1/2}}{q+1}\right)
$$
where $\ff$ is the Fourier transform on $\R$ and
$$
\psi_j(s)=
\frac{q^{1/2}\sin\bigl((j+1)\arccos s\bigr)-q^{-1/2}\sin\bigl((j-1)\arccos s\bigr)}{q+q^{-1}+2(1-2s^2)}
$$
on $(-1,1)$ and $\psi_j=0$ on $\R\setminus(-1,1)$.
\end{remark}

\begin{proof}
Let us fix a ray $w=oy_1y_2\ldots$. Let $k\in\Z$. As we already noticed, if $x\in \hh_k^w$, it has exactly 
one predecessor in $\hh_{k-1}^w$ and $q$ successors in $\hh_{k+1}^w$.
Therefore,
$$
\mathcal{L}\left(q^{-(1/2+is)h_w(x)}\right)=
\left(1-\frac{q^{1/2+is}}{q+1}-q\frac{q^{-1/2-is}}{q+1}\right)q^{-(1/2+is)h_w(x)}.
$$
For a solution $u$ of \eqref{sch}, we consider the Fourier-Helgason transform 
$\tilde{u}(s,w,t)=\ff_\tree[u(\cdot,t)](s,w)$, whose evolution is given by
\[
i\partial_t \tilde{u}=\left(1-\frac{q^{1/2}}{q+1}(q^{is}+q^{-is})\right)\tilde{u}.
\]
Hence, if we set $\sigma=\frac{q^{1/2}}{2(q+1)}$,
\begin{equation}\label{FH}
\tilde{u}(s,w,t)=e^{-i\bigl(1-2\sigma(q^{is}+q^{-is})\bigr)t}\tilde{u}(s,w,0).
\end{equation}

Now we decompose,
\begin{eqnarray*}
\tilde{u}(s,w,t)&=&\sum_{x\in\tree,\ h_w(x)>0}u(x,t)q^{-h_w(x)/2}q^{-is h_w(x)}\\
&&+\sum_{x\in\tree,\ h_w(x)\le0}u(x,t)q^{-h_w(x)/2}q^{-is h_w(x)}\\
&=&\sum_{k=0}^{+\infty}\frac{1}{q^{k/2}}
\left(\sum_{x\in\hh_k^w}u(x,t)\right)\xi^k\\
&&+\sum_{k=1}^{+\infty}q^{k/2}
\left(\sum_{x\in\hh_{-k}^w}u(x,t)\right)\left(\frac{1}{\xi}\right)^k
\end{eqnarray*}
where $\xi=q^{-is}$.

Now write $b_0=1$ and, for $\ell\geq 1$, $b_\ell=\dst\frac{1}{\sqrt{\ell}}\left(\frac{e}{2(q+1)\ell}\right)^{\ell}$ so that if $t_j\in\{0,1\}$,
\begin{eqnarray*}
\abs{\sum_{x\in\hh_k^w}u(x,t_j)}&=&\abs{\sum_{\ell=0}^\infty\sum_{x\in\hh_k^w\cap S_\ell}u(x,t_j)}
\leq \kappa\sum_{\ell=0}^\infty\psi_{\ell,k}b_\ell\\
&\leq&\begin{cases}\dst
\kappa q^k\left(b_k+(q-1)\sum_{p=1}^{\infty}q^{p-1}b_{k+2p}\right)&\mbox{for }k\geq 0\\
\dst\kappa\left(b_{-k}+(q-1)\sum_{p=1}^{\infty}q^{p-1}b_{-k+2p}\right)
&\mbox{for }k\leq -1
\end{cases}.
\end{eqnarray*}

Using that $(k+2p)^{k+2p+1/2}\geq k^{k+1/2}$ when $k,p\geq 1$, and that $(2p)^{2p+1/2}\geq 4$ we get that,
for $k\geq 1$,
\begin{multline*}
\sum_{p=1}^{\infty}q^{p-1}b_{k+2p}=\left(\frac{e}{2(q+1)}\right)^k\frac{1}{q}\sum_{p=1}^{\infty}\left(\frac{e\sqrt{q}}{2(q+1)}\right)^{2p}
\frac{1}{(k+2p)^{k+2p+1/2}}\\
\leq \frac{1}{\sqrt{k}}\left(\frac{e}{2(q+1)k}\right)^k\frac{e^2}{4(q+1)^2\left(1-\frac{e^2q}{4(q+1)^2}\right)}
\leq b_k
\end{multline*}
and the same bound holds for $k=0$, thus
$$
\abs{\sum_{x\in\hh_k^w}u(x,t_j)}\leq \begin{cases}\kappa q&\mbox{for }k=0\\
\kappa q^{k+1}b_k=\dst\kappa q\frac{1}{\sqrt{k}}\left(\frac{eq}{2(q+1)k}\right)^{k}&\mbox{when }k\geq 1\\
\kappa qb_{|k|} =\dst\kappa q\frac{1}{\sqrt{|k|}}\left(\frac{e}{2(q+1)|k|}\right)^{|k|}&\mbox{when }k\leq -1
\end{cases}.
$$
It follows that
$$
\phi^+_j(\xi,w):=\sum_{k=0}^{+\infty}q^{-k/2}
\left(\sum_{x\in\hh_k^w}u(x,t_j)\right)\xi^k
$$
extends into an entire function in $\xi$ of exponential type $\dst \sigma$. Its indicator
diagram $I_j^+$ is therefore included in the closed disc $\bar D(0,\sigma)$.
On the other hand
\begin{equation}\label{exptype}
\phi^-_j(\zeta,w):=
\sum_{k=1}^{+\infty}q^{k/2}
\left(\sum_{x\in\hh_{-k}^w}u(x,t_j)\right)\zeta^k
\end{equation}
extends into an entire function in $\zeta$ of exponential type $\sigma$ as well and its indicator diagram
$I_j^-$ is therefore also included in the disc $\bar D(0,\sigma)$.
Actually, a little more is shown, namely that
\begin{equation}
\label{eq:bounded}
|\phi^\pm_j(\xi,w)|\leq C_q\kappa e^{\sigma|\xi|},
\end{equation}
since we bound the corresponding coefficient of each sum by the $k$-th coefficient of the Taylor series of $e^{\sigma|\xi|}$, fact that motivates the use of the hypothesis \eqref{hyp}.

Let us now turn back to \eqref{FH} which we write as
$$
\tilde u(s,w,t)=e^{-i\bigl(1-2\sigma\bigl(\xi+\xi^{-1})\bigr)t}\bigl(\phi^-_0(\xi^{-1},w)+\phi^+_0(\xi,w)\bigr).
$$
This holds a priori for $\xi=q^{-is}$ and thus extends to $\xi\in\C\setminus\{0\}$ and every $t$.
We write $\tilde u(\xi,w,t)$ for the corresponding extension.

For $t=1$ we obtain
\begin{multline*}
\phi_1^\pm(\xi,w)=-\phi_1^\mp(\xi^{-1},w)\\
+e^{-i}\exp\bigl(2i\sigma(\xi+\xi^{-1})\bigr)\bigl(\phi_0^\pm(\xi,w)+\phi_0^\mp(\xi^{-1},w)\bigr).
\end{multline*}
It follows that $I_1^\pm\subset I_0^\pm+2i\sigma$ which in turn implies that $I_1^\pm$ is reduced to $i\sigma$
and $I_0^\pm$ is reduced to $-i\sigma$.

Let us now take $t=1/2$. Then
$$
\tilde u(\xi,w,1/2)=e^{-i/2}e^{i\sigma(\xi+\xi^{-1})}\bigl(\phi^-_0(\xi^{-1},w)+\phi^+_0(\xi,w)\bigr).
$$
Write $\tilde u(\xi,w,1/2)=u_+(\xi)+u_-(\xi^{-1})$ where $u_+$ (resp. $u_-$) contains all terms of positive (resp. negative)
exponent in the Laurent series of $\tilde u$. The indicator diagram of those functions coincide with $\{0\}$ thus
$u_\pm$ are entire functions of $0$ exponential type.
On the other hand, \eqref{eq:bounded} shows that $u_\pm$ are bounded on $i\R$. Indeed, when $\xi\to+\infty$,
\begin{multline*}
|u_+(i\xi)|\sim |\tilde u(i\xi,w,1/2)|=|e^{-\sigma(\xi+\xi^{-1})}||\phi^-_0(i\xi^{-1},w)+\phi^+_0(i\xi,w)|\\
\sim e^{-\sigma \xi}||\phi^+_0(i\xi,w)|\leq C_q\kappa .
\end{multline*}
To see that $u_+(i\xi)$ is also bounded when $\xi\to-\infty$, let us write
$$
\tilde u(\xi,w,1/2)=e^{-i/2}e^{-i\sigma(\xi+\xi^{-1})}\bigl(\phi^-_1(\xi^{-1},w)+\phi^+_1(\xi,w)\bigr).
$$
Then
\begin{multline*}
|u_+(i\xi)|\sim |\tilde u(i\xi,w,1/2)|=|e^{\sigma(\xi+\xi^{-1})}||\phi^-_1(i\xi^{-1},w)+\phi^+_1(i\xi,w)|\\
\sim e^{\sigma \xi}||\phi^+_1(i\xi,w)|\leq C_q\kappa .
\end{multline*}
The proof for $|u_-(i\xi)|$ is similar but this time $\xi\to 0^\pm$.

Now, according to the Phragmen-Lindel\"{o}f  principle ({\it see e.g.} \cite[Lecture 6]{Lbook}) $u_+$ and $u_-$ are constant 
and thus $\tilde u(\xi,w,1/2)$ does not depend on $\xi$.
It then follows from \eqref{FH} that
$$
\tilde{u}(s,w,0)=C_{w}\exp\bigl(-i\sigma(q^{is}+q^{-is})\bigr)
$$
for some constant $C_{w}$ that depends on the ray $w$.
But, by definition, for $\xi=q^{-is}$
$$
\tilde{u}(\xi,w,0)=\sum_{x\in\tree}u(x,0)\left(\frac{\xi}{\sqrt{q}}\right)^{h_w(x)}
$$
and this extends to all $\xi\in\C\setminus\{0\}$, in particular to $\xi=\sqrt{q}$. This shows that
$$
C_{w}=\exp\bigl(i\sigma(q^{1/2}+q^{-1/2})\bigr)\sum_{x\in\tree}u(x,0)
$$
does not depend on $w$. We thus write $C_{w}=C$.

The integral formula for $u(|x|,0)$ then comes from the inversion formula \eqref{inv} 
 and ({\it see} \cite{cs})
\[
\int_{\partial\tree} q^{-(1/2-is)h_w(x)}d\nu(w)=c(-s)q^{(-is-1/2)|x|}+c(s)q^{(is-1/2)|x|},
\]
where $c(s)=\frac{q^{1/2}}{q+1}\frac{q^{1/2+is}-q^{-1/2-is}}{q^{is}-q^{-is}}$.
\end{proof}

As an immediate corollary, we have the following uniqueness property for strong solutions of \eqref{sch}:

\begin{corollary}\ \\
Assume that $u$ is a strong solution of the equation $\eqref{sch}$. Assume that there exists $\epsilon>0$ and $\kappa$ such that, for $x\not=o$
\[
|u(x,t_i)|\le \frac{\kappa}{\sqrt{|x|}}\left(\frac{e}{(2+\epsilon)(q+1)|x|}\right)^{|x|},\ \ t_0=0,\ t_1=1.
\]
Then $u\equiv0$.
\end{corollary}

\begin{remark}
Note that $\tilde{u}(s,w,0)=C\exp\bigl(-i\sigma(q^{is}+q^{-is})\bigr)$ and therefore $\tilde{u}(s,w,t)=C\exp{\bigl(-i\big(t+\sigma(1-2t)(q^{is}+q^{-is})\big)\bigr)}$. Applying the inversion formula \eqref{inv} to this function we get an integral representation formula for the evolution $u(x,t)$.

This could also be obtained by analytic continuation of the well-known solution of the heat equation on the tree.
\end{remark}

\begin{remark}\label{rem:freeshrodscale} We leave as an exercise to the reader to check that, if $u$ is a strong solution of the equation 
$i\partial_tu(x,t)=\lambda\mathcal{L}u(x,t)$, with $u(x,0)=u_0(x)$, $\lambda>0$ and if
\begin{equation}
\label{eq:freeshrodscale}
|u(x,t_i)|\le \frac{\kappa}{\sqrt{|x|}}\left(\frac{e\lambda}{2(q+1)|x|}\right)^{|x|}
\end{equation}
then
\[
u_0(x)=\frac{C}{q^{|x|/2}}\int_{0}^\pi \exp\left(-i\frac{q^{1/2}\lambda}{q+1}\cos(z)\right)
\ffi_{|x|}(z)\sin(z)\d z
\]
for some constant $C$. 

Note that when $\lambda=q+1$, the condition \eqref{eq:freeshrodscale} is the same for the tree $\tree_q$ and for $\Z$
so that the dependence on the tree is hidden.

\end{remark}

\begin{remark}
The only other uncertainty principle on the tree we are aware of is due to Astengo \cite{Ast}. 
It is of a rather different nature to our results.
More precisely, Astengo states an uncertainty principle in terms of a function on the tree and the modulus of
its Fourier-Helgason transform. In view of Formula \eqref{FH}, Astengo's result immediately translates into a result for
solutions of the free Shr\"odinger equation on the tree:

{\sl Assume that $u$ is a strong solution of the equation 
\begin{equation}
i\partial_tu(x,t)=\mathcal{L}u(x,t)=u(x,t)-\frac{1}{q+1}\sum_{y\sim x}u(y,t),\ \ x\in\tree
\end{equation}
with $u(x,0)=u_0(x)$. Assume that there is a time $t_0$ such that
\begin{enumerate}
\renewcommand{\theenumi}{\roman{enumi}}
\item $|u_0(x)|\leq Ce^{-\alpha|x|}$ for some $C>0$ and some $\alpha>\frac{1}{2}\log q$;

\item $s\to\norm{\ff_{\tree} u(\cdot,s,t_0)}_{L^2(\partial\tree)}\in L^1(\T)$;

\item $\int_{\T}\log\norm{\ff_{\tree} u(\cdot,s,t_0)}_{L^2(\partial\tree)}\,\frac{\mathrm{d}s}{s}<+\infty$;
\end{enumerate}
then $u=0$.}
\end{remark}

\section{Uniqueness for perturbed problems using Carleman estimates}
\label{carleman}

In this section we consider the problem 
\begin{equation}\label{schr}
\partial_tu=i(\mathcal{L}u+Vu)
\end{equation}
where $V=V(x,t)$ is a bounded potential. 

We are going to begin this section by pointing out that a fast decaying solution at times $t=0$ and $t=1$ extends the fast decay to the whole interval $[0,1]$. This is given by an immediate extension of part of the results in \cite{jlmp}.
For convenience, the equation is written in a different way. In any case, by doing a suitable change of variables one can see that the results described in this section can be rewritten in terms of a solution of $i\partial_tu=\mathcal{L}u+Vu$.
We first need an auxiliary lemma:

\begin{lemma}\ \\
Let $u\in C^1([0,T],\tree)$ satisfy \eqref{schr} where $V$ is a complex valued functions in $\tree\times[0,T]$ and bounded. Let
$$
\psi_\alpha(x,t)=(1+|x|)^{\alpha|x|/(1+t)}, \ \ \alpha\in(0,1].
$$
Then, for $T>0$,
$$
\|\psi_\alpha(T)u(T)\|_2^2\le e^{CT}\|\psi_\alpha(0)u(0)\|_2^2,
$$
provided the right-hand side is finite.
\end{lemma}

\begin{remark} This is a tree analogue of \cite[Proposition 3.1]{jlmp}. We may as well consider
the more general equation
$$
\partial_t u(x,t)=i\big(\mathcal{L}u(x,t)+V(x,t)u+F(x,t)\big),
$$
where $V$ and $F$ are complex valued functions in $\tree\times[0,T]$ and bounded. In this case,
a simple adaptation of the proof below shows that
$$
\|\psi_\alpha(T)u(T)\|_2^2\le e^{CT}\left(\|\psi_\alpha(0)u(0)\|_2^2+\int_0^T\|\psi_\alpha(s)F(s)\|_2^2\d s\right),
$$
provided the right-hand side is finite.
\end{remark}

\begin{proof}
Define $f(x,t)=\psi_\alpha(x,t)u(x,t)$ and $H(t)=\|f(t)\|_2^2$ for a fixed $\alpha$. We will just write $\psi=\psi_\alpha$. Notice that $\psi$ only depends on $|x|$, so for $|x|=n$ we write $\psi(x)=\psi(n)$.

Formally,
$$
\partial_t f=i\psi\mathcal{L}(\psi^{-1}f)+\phi_t f+iVf=\mathcal{S}f+\mathcal{A}f+iVf,
$$
where $\phi=\log\psi$ and
\begin{eqnarray*}
\mathcal{S}f&=&\phi_tf+\frac{i}{q+1}\sum_{y\sim x}\sinh(\phi(x,t)-\phi(y,t))f(y)\\
\mathcal{A}f&=& \frac{i}{q+1}\sum_{y\sim x}\cosh(\phi(x,t)-\phi(y,t))f(y)-if(x).
\end{eqnarray*}
are symmetric and skew-symmetric operators respectively. Since 
$$
\partial_t H(t)=2\Re\langle \partial_tf,f\rangle,
$$
it is easy to check that $\partial_t H(t)$ is
\begin{eqnarray*}
&\le& \|V\|_\infty\|f\|_2\\
&&+\left(2\phi_t(0)+\frac{2}{\sqrt{q}}|\sinh(\phi(1)-\phi(0)|)\right)|f(o)|^2\\
&&+\sum_{n\ge 1,|x|=n}\left(2\phi_t(n)+\frac{2\sqrt{q}}{q+1}\big|\sinh\big(\phi(n)-\phi(n-1)\big)\big|\right)|f(x)|^2\\
&&+\frac{2\sqrt{q}}{q+1}\sum_{n\ge 1,|x|=n}\big|\sinh\big(\phi(n+1)-\phi(n)\big)\big||f(x)|^2.
\end{eqnarray*}

The result follows after proving that the last three terms are bounded by $C\|f\|_2$, in the same spirit as in \cite{jlmp}. To justify this formal argument, we can prove again the same result (now rigorously) for a truncated weight $\psi_N$ and then let $N\to\infty$  (See \cite{jlmp} for this argument in the line).
\end{proof}

This result shows that if we have a solution of $\eqref{schr}$ with fast decay at time $t=0$, the solution has fast decay at any future time, although the decay gets worse with time. Our aim now is to use also the fast decay at time $t=1$ to improve the decay at future times.

\begin{proposition}\label{prop:4.3}\ \\
Let $\gamma>0$ and $V$ a bounded potential. Let $u$ be a strong solution of \eqref{schr}
and assume that at times $t=0$ and $t=1$,
$$
\|(1+|x|)^{\gamma(1+|x|)}u(x,t)\|_2<+\infty,\ \ t\in\{0,1\}.
$$
Then, for all $t\in[0,1],\ \|(1+|x|)^{\gamma(1+|x|)}u(t)\|_2<+\infty.$
\end{proposition}

\begin{remark}
This is the tree analogue of \cite[Proposition 4.1]{jlmp} on $\Z$.
\end{remark}

\begin{proof}
For $1/2<b<1$, let $\phi_b(n)=\gamma(1+n)\log^b(1+n),\ n\in\mathbb{N}\cup\{0\}$.
Set $f=e^{\phi_b(|x|)}u$ and, as before $H(t)=\|f(t)\|_2^2$. The  previous lemma shows that $H(t)$ is finite for all $t$, so the subsequent formal computations are justified. We will show that, for some $C>0$,
\begin{eqnarray*}
H_b(t)&\le&e^{Ct(1-t)}H_b(0)^{1-t}H_b(1)^t\\
&\le&e^{Ct(1-t)} \|(1+|x|)^{\gamma(1+|x|)}u(0)\|_2^{2(1-t)} \|(1+|x|)^{\gamma(1+|x|)}u(1)\|_2^{2t}.
\end{eqnarray*}
The result will follow by letting $b\to1$ and applying the monotone convergence theorem.

In order to prove our claim, we write again $\partial_t f=\mathcal{S}f+\mathcal{A}f+iVf$ and, as shown in \cite{jlmp}, the claim follows from a lower bound
\begin{equation}
\label{eqconm}
\langle [\mathcal{S},\mathcal{A}]f,f\rangle \ge -C\|f\|^2,
\end{equation}
with $\mathcal{S},\mathcal{A}$ the operators defined in the previous lemma, in this case for the weight $e^{\phi_b}$. 
Since $\phi_b$ does not depend on $t$, it is easy to check that
$(q+1)^2 \langle[\mathcal{S},\mathcal{A}]f,f\rangle$ is
\begin{eqnarray*}
&=&\dst\sum_{x\in\tree}\sum_{y\sim x}\sum_{z\sim y}\sinh\big(2\phi_b(|y|)-\phi_b(|x|)-\phi_b(|z|)\big)f(z)\overline{f(x)}\\
&=&\sinh\big(2\phi_b(1)-2\phi_b(0)\big)|f(o)|^2\\
&&\!+2\sinh\big(2\phi_b(1)-\phi_b(0)-\phi_b(2)\big)\Re\sum_{z\in o_{ff}}f(z)\overline{f(o)}\\
&&\!+\sum_{x\in\tree\setminus\{o\}}\!\sinh\big(2\phi_b(|x|-1)-2\phi_b(|x|)\big)\!\sum_{z\in x_{pf}}\!f(z)\overline{f(x)}\\
&&\!+\sum_{x\in\tree\setminus\{o\}}\!\sinh\big(2\phi_b(|x|-1)-2\phi_b(|x|)\big)|f(x)|^2\\
&&\!+2\Re\sum_{x\in\tree\setminus\{o\}}\!\sinh\big(2\phi_b(|x|+1)-\phi_b(|x|)-\phi_b(|x|+2)\big)\!\sum_{z\in x_{ff}}\!f(z)\overline{f(x)}\\
&&\!+\sum_{x\in\tree\setminus\{o\}}\!q\sinh\big(2\phi_b(|x|+1)-2\phi_b(|x|)\big)|f(x)|^2\\
&=&S_1+\cdots+S_6.
\end{eqnarray*}

As for each $n$, there exists $\gamma_n$ such that, for every $1/2<b<1$, $|\Phi_b(n)|\leq\gamma_n$, there exists
a constant $C$ such that
$S_1,S_2\geq -C\norm{f}^2$.

As in \cite{jlmp}, there exists a constant $\kappa$ such that, for every $n$,
$|\sinh\big(2\phi_b(n+1)-\phi_b(n)-\phi_b(n+2)\big)|\leq\kappa$. Further, $|x_{ff}|= q^2$ so that Cauchy-Schwarz shows that there is
a constant $C$ such that $S_5\geq -C\norm{f}^2$.

Next, if $|x|\ge 2$,
\begin{eqnarray*}
\abs{\sum_{z\in x_{pf}}f(z)\overline{f(x)}}&\leq& \frac{1}{2}\sum_{z\in x_{pf}}\bigl( |f(z)|^2+|f(x)|^2\bigr)\\
&=&\frac{q-1}{2}|f(x)|^2+\frac{1}{2}\sum_{z\in x_{pf}}|f(z)|^2,
\end{eqnarray*}
while if $|x|=1$,
\begin{eqnarray*}
\abs{\sum_{z\in x_{pf}}f(z)\overline{f(x)}}\le\frac{q}{2}|f(x)|^2+\frac{1}{2}\sum_{z\in x_{pf}}|f(z)|^2.
\end{eqnarray*}
But then
\begin{eqnarray*}
S_3&\geq& -\frac{1}{2}\sum_{x\in\tree\setminus\{o\}}\sinh\big(2\phi_b(|x|)-2\phi_b(|x|-1)\big)\sum_{z\in x_{pf}}|f(z)|^2\\
&&-\frac{q-1}{2}\sum_{|x|\ge 2}\sinh\big(2\phi_b(|x|)-2\phi_b(|x|-1)\big)|f(x)|^2\\
&&-\frac{q}{2}\sum_{|x|=1}\sinh\big(2\phi_b(1)-2\phi_b(0)\big)|f(x)|^2\\
&=&-(q-1)\sum_{x\in\tree\setminus\{o\}}\sinh\big(2\phi_b(|x|)-2\phi_b(|x|-1)\big)|f(x)|^2\\
&&-\sum_{|x|=1}\sinh\big(2\phi_b(1)-2\phi_b(0)\big)|f(x)|^2\\
&=&S_3^a+S_3^b
\end{eqnarray*}
since each $x\in\tree\setminus\{o\}$ appears $q-1$ or $q$ times in the first sum if $|x|\ge2$ or $|x|=1$.
It follows that $S_3^b\ge -C\|f\|^2$ and
$$
S_3^a+S_4+S_6\geq
q\sum_{x\in\tree\setminus\{o\}}\psi_b(|x|)|f(x)|^2
\geq 0
$$
where
$$
\psi_b(n)=\sinh\big(2\phi_b(n+1)-2\phi_b(n)\big)-\sinh\big(2\phi_b(n)-2\phi_b(n-1)\big)\geq 0
$$
due to the properties of the function $(1+x)\log^b(1+x)$ for $x>0$ and $1/2<b<1$,
{\it see} \cite{jlmp}.
\end{proof}

As it happens in the continuous case, or in $\Z^d$, uniqueness holds from an argument related to Carleman inequalities. Here we prove the following Carleman inequality:

\begin{lemma}[Carleman inequality on the tree]\ \\
Let $\varphi:[0,1]\rightarrow\mathbb{R}$ be a smooth function, $\beta>0$ and $\gamma>\frac{1}{2\beta}$. There exists $R_0=R_0(\|\varphi\|_\infty+\|\varphi'\|_\infty+\|\varphi''\|_\infty,\beta,\gamma)$ such that, if $R>R_0$, $\alpha\ge \gamma R\log R$ and if
$g$ is a function on $\tree\times[0,1]$, $g\in C_0^1([0,1],\ell^2(\tree))$ has its support contained in the set
\[
\{(x,t):|x|/R+\varphi(t)\ge \beta\},
\]
then
\begin{multline*}
\sinh\frac{2\alpha}{R^2}\cosh\frac{4\alpha\beta}{R}\|e^{\alpha\left(\frac{|x|}{R}+\varphi\right)^2}g\mathbf{1}_{|x|\geq1}\|_{L^2_{x,t}}^2\\
\le (q+1)^2\|e^{\alpha\left(\frac{|x|}{R}+\varphi(t)\right)^2}(i\partial_t+\mathcal{L})g\|_{L^2_{x,t}}^2\\
+\int_0^1\sinh\frac{4\alpha}{R}\left(\frac{1}{2R}+\ffi\right)
\sum_{|x|=1}\abs{e^{\alpha\left(\frac{1}{R}+\varphi(t)\right)^2}g(x)}^2\d t.
\end{multline*}
\end{lemma}

\begin{remark}
In the case of $\Z$ or, in general, of $\Z^d$, where the combinatorics makes the study of the problem much easier this corresponds to \cite[Lemma 2.1]{fbv}. Further, on the tree, the inequality contains an extra-term. Fortunately, this term will be harmless.
\end{remark}

\begin{proof} Let $\phi$ be defined by $\phi(n)=\alpha\dst\left(\frac{n}{R}+\ffi(t)\right)^2$.
For $f=e^{\phi}g$ we have,
$$
e^{\phi}(i\partial_t+\mathcal{L})g=\mathcal{S}f+\mathcal{A}f,
$$
where
 \begin{eqnarray*}
 \mathcal{S}f&=& i\partial_tf+\frac{1}{q+1}\sum_{y\sim x}\cosh(\phi(x,t)-\phi(y,t))f(y,t)-f,\\
 \mathcal{A}f&=& -i\phi_tf+\frac{1}{q+1}\sum_{y\sim x}\sinh(\phi(x,t)-\phi(y,t))f(y,t).
 \end{eqnarray*}

We need to give a lower bound for the commutator, which immediately implies the result using the fact that
$$
\|e^{\alpha\left(\frac{|x|}{R}+\varphi(t)\right)^2}(i\partial_t+\mathcal{L})g\|_{L^2_{x,t}}^2\ge  \langle[\mathcal{S},\mathcal{A}]f,f\rangle.
$$

To simplify notation, we will not explicitly write the dependence of $f$ on the time variable $t$ so that $f(x)$ means $f(x,t)$,
$x\in\tree$, $t\in[0,1]$.
 A simple computation shows that
\begin{equation}
\label{cgena}
 \langle[\mathcal{S},\mathcal{A}]f,f\rangle=\int_0^1 S(t)\,\mbox{d}t\
\end{equation}
where
 \begin{multline}
S(t):=\dst\sum_{x\in\tree}\phi_{tt}(x)|f(x)|^2\\
+\frac{2}{q+1}\sum_{x\in\tree}\sum_{y\sim x}(\phi_t(x)-\phi_t(y))\cosh(\phi(x)-\phi(y))f(y)\overline{f(x)}\\
+\dst\frac{1}{(q+1)^2}\sum_{x\in\tree}\sum_{y\sim x}\sum_{z\sim y}\sinh(2\phi(y)-\phi(x)-\phi(z))f(z)\overline{f(x)}.
\label{cgen}
 \end{multline}

As in the previous proof, we split them into sums over mothers and daughters. 
Recall that the root has only daughters while the rest
of the points in the tree have a single mother and $q$ daughters. Further,
the function $\phi(x,t)$ only depends on $|x|$, 
the distance of a point in the tree to the root $o$. We therefore decompose
the sums in \eqref{cgen} as follows: $S(t)=S_1+\cdots+S_7$ where

--- The first sum in \eqref{cgen} is $\dst S_1=\sum_{n\ge0}\sum_{|x|=n}\phi_{tt}(n)|f(x)|^2$.

--- For the second sum in \eqref{cgen}, each pair $x\sim y$ appears twice, once $|x|=|y|+1$, once with
$|x|=|y|-1$. Therefore, this sum can be rewritten as
$$
S_2=\frac{4}{q+1}\Im\sum_{n\ge1}\sum_{|x|=n}(\phi_t(n)-\phi_t(n-1))\cosh(\phi(n)-\phi(n-1))f(x_p)\overline{f(x)}.
$$

--- For the last sum in \eqref{cgen}, we need to distinguish more cases:

a) $x=o$, $y$ any daughter and $z=o$. This happens $q+1$ times and leads to
$$
S_3=\frac{1}{q+1}\sinh 2\bigl(\phi(1)-\phi(0)\bigr)|f(o)|^2;
$$

b) $x\in\tree\setminus\{o\}$, {\it i.e.} $n:=|x|\geq 1$ $y$ is one of the $q$ daughters of $x$ and $z=x$
which leads to
$$
S_4=\frac{q}{(q+1)^2}\sum_{n\ge1}\sum_{|x|=n}\sinh2\bigl(\phi(n+1)-\phi(n)\bigr)|f(x)|^2
$$
while if $y$ is the mother of $x$ and $z=x$, we get
$$
S_5=\frac{1}{(q+1)^2}\sum_{n\ge1}\sum_{|x|=n}\sinh2\bigl(\phi(n-1)-\phi(n)\bigr)|f(x)|^2;
$$

c) $x\in\tree\setminus\{o\}$, {\it i.e.} $n:=|x|\geq 1$, $y$ is the mother of $x$ and $z$ is any
of the sisters of $x$, we get
$$
S_6=\frac{1}{(q+1)^2}\sum_{n\ge1}\sum_{|x|=n}\sum_{z\in x_{pf}}\sinh2\bigl(\phi(n-1)-\phi(n)\bigr)f(z)\overline{f(x)};
$$

--- Finally, for all other terms $x$ is the grand-mother of $z$ and each such couple $(x,z)$ appears twice.
As $|z|\geq 2$, this may be written as
$$
S_7=\frac{2}{(q+1)^2}\Re\sum_{n\ge2}\sum_{|x|=n}\sinh\bigl(2\phi(n-1)-\phi(n)-\phi(n-2)\bigr)f(x)\overline{f(x_{pp})}.
$$

Before estimating those quantities, as $\phi(n)=\alpha\left(\frac{n}{R}+\varphi(t)\right)^2$, we obtain
\begin{eqnarray*}
\phi_t(n)&=&2\alpha\left(\frac{n}{R}+\varphi\right)\varphi'\\
\phi_{tt}(n)&=&2\alpha\left[\left(\frac{n}{R}+\varphi\right)\varphi''+(\varphi')^2\right]\\
\phi_{t}(n)-\phi_{t}(n-1)&=&\frac{2\alpha}{R}\varphi'\\
\phi(n-1)-\phi(n)&=&-\frac{2\alpha}{R}\left(\frac{n-1/2}{R}+\varphi\right)\\
\phi(n)+\phi(n+2)-2\phi(n+1)&=&\frac{2\alpha}{R^2}.\\
\end{eqnarray*}

Let us now estimate $S_1$ to $S_7$. We will treat them from the simplest to the most involved one rather than the order in which they appeared in the above decomposition. We start with $S_1$, which can be bounded by
\begin{equation}
\label{eq:estS1}
S_1\geq -2\|\ffi''\|_\infty\alpha\sum_{n\geq0}\abs{\frac{n}{R}+\ffi}\sum_{|x|=n}|f(x)|^2.
\end{equation}

To estimate $S_7$, we write $2\Re(f(x)\overline{f(x_{pp})}=-|f(x)-f(x_{pp})|^2+|f(x)|^2+|f(x_{pp})|^2$, then
\begin{eqnarray}
S_7&=&\frac{\sinh\frac{2\alpha}{R^2}}{(q+1)^2}\sum_{n\ge2}\sum_{|x|=n}\bigl(|f(x)-f(x_{pp})|^2-|f(x)|^2-|f(x_{pp})|^2\bigr)\nonumber\\
&\geq&-\frac{\sinh\frac{2\alpha}{R^2}}{(q+1)^2}\left(\sum_{n\ge2}\sum_{|x|=n}|f(x)|^2+\sum_{n\ge2}\sum_{|x|=n}|f(x_{pp})|^2\right)\nonumber\\
&\geq&-\frac{\sinh\frac{2\alpha}{R^2}}{(q+1)^2}\left(q(q+1)|f(o)|^2+q^2\sum_{|x|=1}|f(x)|^2\right.\nonumber\\
&&\qquad\left.+(q^2+1)\sum_{n\ge2}\sum_{|x|=n}|f(x)|^2\right)\label{eq:estS7}
\end{eqnarray}
since $o$ has $q(q+1)$ grand-daughters, it appears $q(q+1)$ times as an $x_{pp}$,
if $|x|\geq 1$, it has $q^2$ grand-daughters and thus will appear $q^2$ times in the second sum.

Next, for $S_6$, we use that $f(z)\overline{f(x)}\ge -\frac{1}{2}(|f(x)|^2+|f(z)|^2)$ to obtain
\begin{eqnarray*}
S_6&\geq& -\frac{1}{2(q+1)^2}\sum_{n\ge1}\abs{\sinh2\bigl(\phi(n-1)-\phi(n)\bigr)}\times\\
&&\qquad\qquad\times\sum_{|x|=n}\sum_{z\in x_{pf}}(|f(x)|^2+|f(z)|^2)\\
&=&-\frac{1}{(q+1)^2}\left(
q\abs{\sinh2\bigl(\phi(0)-\phi(1)\bigr)}\sum_{|x|=1}|f(x)|^2\right.\\
&&\qquad\qquad\left.+(q-1)\sum_{n\ge1}\abs{\sinh2\bigl(\phi(n-1)-\phi(n)\bigr)}\sum_{|x|=n}|f(x)|^2\right).
\end{eqnarray*}
Here we use the fact that $x_{pf}$ has $q$ elements if $|x|=1$ and $q-1$ elements otherwise 
for $\sum_{|x|=n}\sum_{z\in x_{pf}}|f(x)|^2$ and we use \eqref{eq:usefull} for the second sum.
Finally, using the expression of $\phi$, we get
\begin{multline}
S_6\geq 
-\frac{1}{(q+1)^2}\left(
q\sinh\frac{4\alpha}{R}\left(\frac{1}{2R}+\ffi\right)\sum_{|x|=1}|f(x)|^2\right.\\
\left.+(q-1)\sum_{n\ge2}\sinh\frac{4\alpha}{R}\left(\frac{n-1/2}{R}+\ffi\right)\sum_{|x|=n}|f(x)|^2\right).
\label{eq:estS6}
\end{multline}

Now, for $S_2$, let us first introduce
$$
\Psi(n)=\cosh\bigl(\phi(n)-\phi(n-1)\bigr)
$$
and
$$
\Sigma_n=\sum_{|x|=n}\bigl(q^{1/2}|f(x)|^2+q^{-1/2}|f(x_p)|^2\bigr).
$$
We use the expression of $\phi_t$ and the fact that
$$
2|f(x)f(x_p)|\le q^{1/2}|f(x)|^2+q^{-1/2}|f(x_p)|^2
$$
to bound $S_2$ by
\begin{eqnarray*}
&\geq&\!-\frac{4\alpha|\ffi'|}{(q+1)R}\sum_{n\geq 1}\Psi(n)\Sigma_n\\
&=&\!-\frac{4\alpha|\ffi'|}{q^{1/2}R}\Psi(1)|f(o)|^2-\frac{4q^{1/2}\alpha |\ffi'|}{(q+1)R}
\sum_{n\geq 1}\sum_{|x|=n}\bigl[\Psi(n)+\Psi(n+1)\bigr]|f(x)|^2
\end{eqnarray*}
since $o$ will appear $q+1$ times as an $x_p$ and each $x$ with $|x|\geq 1$ will appear once as an $x$ and $q$ times as an $x_p$.
Using the expression of $\phi$ we conclude that
\begin{multline}
\label{eq:estS2}
S_2\geq -\frac{4\alpha\|\ffi'\|_\infty}{Rq^{1/2}}\cosh\frac{2\alpha}{R}\left(\frac{1}{2R}+\ffi\right)|f(o)|^2\\
-\frac{4q^{1/2}\alpha\|\ffi'\|_\infty}{(q+1)R}\cosh\frac{\alpha}{R^2}
\sum_{n\geq 1}\cosh\frac{2\alpha}{R}\left(\frac{n}{R}+\ffi\right)\sum_{|x|=n}|f(x)|^2.
\end{multline}

Next, write
\begin{eqnarray*}
S_4&=&\frac{q-1}{(q+1)^2}\sum_{n\ge1}\sum_{|x|=n}\sinh2\bigl(\phi(n+1)-\phi(n)\bigr)|f(x)|^2\\
&&+\frac{1}{(q+1)^2}\sum_{n\ge1}\sum_{|x|=n}\sinh2\bigl(\phi(n+1)-\phi(n)\bigr)|f(x)|^2\\
&=&S_4^a+S_4^b.
\end{eqnarray*}
We will group $S_4^b$ and $S_5$ noticing that
\begin{multline*}
\sinh\frac{4\alpha}{R}\left(\frac{n+1/2}{R}+\varphi\right)-\sinh\frac{4\alpha}{R}\left(\frac{n-1/2}{R}+\varphi\right)\\
=2\cosh\frac{4\alpha}{R}\left(\frac{n}{R}+\varphi\right)\sinh\frac{2\alpha}{R^2}.
\end{multline*}
This leads to
\begin{equation}
\label{eq:estS4S5}
S_4^b+S_5
\geq
\frac{2}{(q+1)^2}\sinh\frac{2\alpha}{R^2}\sum_{n\ge1}\cosh\frac{4\alpha}{R}\left(\frac{n}{R}+\varphi\right)\sum_{|x|=n}|f(x)|^2.
\end{equation}

We are now in position to estimate $S_1+\cdots+S_7$.
Let us first isolate all terms containing $|f(o)|^2$.
They appear in $\eqref{eq:estS1}$, $S_3$, \eqref{eq:estS7} and \eqref{eq:estS2}.

The factor of $|f(o)|^2$ is
\begin{multline*}
A:=-2\alpha\|\ffi\|_\infty\|\ffi''\|_\infty
+\frac{1}{q+1}\sinh \frac{4\alpha}{R}\left(\frac{1}{2R}+\ffi\right)\\
-\frac{q\sinh\frac{2\alpha}{R^2}}{q+1}
-\frac{4\alpha\|\ffi'\|_\infty}{Rq^{1/2}}\cosh\frac{2\alpha}{R}\left(\frac{1}{2R}+\ffi\right).
\end{multline*}

Now, the hypothesis of the lemma show that, if $f(o)\not=0$, then $\ffi\geq \beta>0$.
Further, as $\alpha>\dst\frac{1}{2\beta}R\log R$, it is easy to see that the dominating term in $A$
is the second one and that the other three can be absorbed in it provided $R$ is large enough.
Thus $A\geq 0$ 
if $R$ is large enough
(depending on $q$, $\|\ffi\|_\infty,\|\ffi'\|_\infty,\|\ffi''\|_\infty$ and $\beta$).

Next, we compute the factor of $\dst \sum_{|x|=1}|f(x)|^2$. The one stemming from $S_4^a$ and the one appearing in \eqref{eq:estS6}
give
\begin{multline*}
\ent{\frac{q-1}{(q+1)^2}\sinh\frac{4\alpha}{R}\left(\frac{3}{2R}+\ffi\right)
-\frac{q}{(q+1)^2}\sinh\frac{4\alpha}{R}\left(\frac{1}{2R}+\ffi\right)}\\
\geq-\frac{1}{(q+1)^2}\sinh\frac{4\alpha}{R}\left(\frac{1}{2R}+\ffi\right).
\end{multline*}
The remaining terms for $|x|=1$ come from \eqref{eq:estS4S5}, \eqref{eq:estS2},
\eqref{eq:estS7} and \eqref{eq:estS1}. The factor of $\dst \sum_{|x|=1}|f(x)|^2$ stemming from those terms is
\begin{multline*}
\frac{2}{(q+1)^2}\sinh\frac{2\alpha}{R^2}\cosh\frac{4\alpha}{R}\left(\frac{1}{R}+\varphi\right)\\
-\frac{4q^{1/2}\alpha\|\ffi'\|_\infty}{(q+1)R}\cosh\frac{\alpha}{R^2}
\cosh\frac{2\alpha}{R}\left(\frac{1}{R}+\ffi\right)\\
-\frac{q^2\sinh\frac{2\alpha}{R^2}}{(q+1)^2}
-2\|\ffi''\|_\infty\alpha\abs{\frac{1}{R}+\ffi}.
\end{multline*}
The three last terms are again absorbed in the first one ({\it see} \cite{fbv} for details).
We are thus left with
\begin{multline*}
\frac{1}{(q+1)^2}\sinh\frac{2\alpha}{R^2}\cosh\frac{4\alpha}{R}\left(\frac{1}{R}+\varphi\right)
\sum_{|x|=1}|f(x)|^2\\
\geq \frac{1}{(q+1)^2}\sinh\frac{2\alpha}{R^2}\cosh\frac{4\alpha\beta}{R}
\sum_{|x|=1}|f(x)|^2
\end{multline*}
because of the support property of $f$.

For $n\geq 2$ the factor of $\dst \sum_{|x|=n}|f(x)|^2$ come from

--- first those from $S_4^a$ and from \eqref{eq:estS6} which now are
$$
\frac{q-1}{(q+1)^2}\sinh\frac{4\alpha}{R}\left(\frac{n+1/2}{R}+\ffi\right)
-\frac{q-1}{(q+1)^2}\sinh\frac{4\alpha}{R}\left(\frac{n-1/2}{R}+\ffi\right)
\geq0,
$$

--- the remaining ones coming from \eqref{eq:estS1}, \eqref{eq:estS7}, \eqref{eq:estS2} and \eqref{eq:estS4S5}
\begin{multline*}
-2\|\ffi''\|_\infty\alpha\abs{\frac{n}{R}+\ffi}-\frac{q^2+1}{(q+1)^2}\sinh\frac{2\alpha}{R^2}\\
-\frac{4q^{1/2}\alpha\|\ffi'\|_\infty}{(q+1)R}\cosh\frac{\alpha}{R^2}\cosh\frac{2\alpha}{R}\left(\frac{n}{R}+\ffi\right)\\
+\frac{2}{(q+1)^2}\sinh\frac{2\alpha}{R^2}\cosh\frac{4\alpha}{R}\left(\frac{n}{R}+\varphi\right).
\end{multline*}
The first three terms are again absorbed in the last one ({\it see} \cite{fbv} for details).
We are thus left with
\begin{multline*}
\frac{1}{(q+1)^2}\sinh\frac{2\alpha}{R^2}\sum_{n\geq 2}\cosh\frac{4\alpha}{R}\left(\frac{n}{R}+\varphi\right)
\sum_{|x|=n}|f(x)|^2\\
\geq \frac{1}{(q+1)^2}\sinh\frac{2\alpha}{R^2}\cosh\frac{4\alpha\beta}{R}\sum_{n\geq 2}\sum_{|x|=n}|f(x)|^2
\end{multline*}
because of the support property of $f$.

In summary, if $R$ is large enough,
\begin{multline*}
\langle [\mathcal{S},\mathcal{A}]f,f\rangle
\geq 
-\int_0^1\frac{1}{(q+1)^2}\sinh\frac{4\alpha}{R}\left(\frac{1}{2R}+\ffi\right)
\sum_{|x|=1}|f(x)|^2\,\mbox{d}t\\
+\frac{1}{(q+1)^2}\sinh\frac{2\alpha}{R^2}\cosh\frac{4\alpha\beta}{R}\int_0^1\sum_{n\geq 1}\sum_{|x|=n}|f(x)|^2\,\mbox{d}t
\end{multline*}
as announced.
\end{proof}

Even though we need a correction term in order to give the Carleman estimate, we can adapt the argument of the proof of
\cite[Theorem 1.1]{fbv} to give again a lower bound for solutions of Schr\"odinger evolutions on trees.

\begin{theorem}[Lower bound for solutions of Schr\"odinger equations]\label{th:4.7}\ \\
Let $q\geq 2$, $A,L,\eta>0$ then there exists $R_0=R_0(q,A,L)>0$ and $c=c(q,\eta)$ such that

--- if $V$ is a bounded function on $\tree$ with
\[
\|V\|_{\infty}=\sup_{t\in[0,1],x\in\tree}\{|V(x,t)|\}\le L,
\]

--- and $u\in C^1([0,1]:\ell^2(\tree))$ is a strong solution of
$$
\partial_t u=i(\lll u+Vu)
$$
that satisfies the bounds
\[
\int_0^1\sum_{x\in\tree}|u(x,t)|^2\,\d t\le A^2\quad,\quad
\int_{1/2-1/8}^{1/2+1/8}|u(x_0,t)|^2\,\d t\ge1
\]
for some $x_0$ with $|x_0|=2$.

Then for $R\ge R_0$,
\[
\lambda(R)\equiv \left(\int_{0}^{1}\sum_{\lfloor R\rfloor-1\le |x|\le \lfloor R\rfloor+1}|u(x,t)|^2\,\d t\right)^{1/2}\ge c e^{-(1+\eta)R \log R}.
\]
\end{theorem}

\begin{proof} 
For $\epsilon>0$ fixed let us define the following cut-off functions:

--- we define $\theta^R,\mu$ to be $C^\infty(\mathbb{R})$ functions such that $0\le \theta^R,\mu\le 1$
and
\begin{equation}\label{cutoffa}
\theta^R(x)=\left\{\begin{array}{ll}1,&|x|\le R-1\\0,&|x|\ge R\end{array}\right.\quad \mu(x)=\left\{\begin{array}{ll}1,&|x|\ge \epsilon^{-1}+1\\0,&|x|\le \epsilon^{-1}\end{array}\right..
\end{equation}

--- and a $C^\infty([0,1])$ function $\varphi$ such that $0\le \varphi\le 2+\epsilon^{-1}$ and
\begin{equation}\label{cutoffb}
\varphi(t)=\left\{\begin{array}{ll}2+\epsilon^{-1},&t\in[\frac12-\frac18,\frac12+\frac18]\\0,&t\in [0,\frac14]\cup[\frac34,1]
\end{array}\right..
\end{equation}

We are going to apply the previous lemma to
$$
g(x,t):=\theta^R(|x|)\mu\left(\frac{|x|}{R}+\varphi(t)\right)u(x,t),\ \ x\in\tree,\ \ t\in[0,1].
$$

Notice that the evolution of $g$ is given by the expression
\begin{eqnarray*}
(i\partial_t+\mathcal{L})g\!&=&\!
\theta^R\mu\left(\frac{|x|}{R}+\varphi\right)(i\partial_t u+\lll u)
+i\varphi' \theta^R(x)\mu'\left(\frac{|x|}{R}+\varphi\right)u\\
&&+\theta^R(x)\frac{1}{q+1}\sum_{y\sim x}\left(\mu\left(\frac{|y|}{R}+\varphi\right)-\mu\left(\frac{|x|}{R}+\varphi\right)\right)u(y,t)\\
&&+\frac{1}{q+1}\sum_{y\sim x}\left(\theta^R(|y|)-\theta^R(|x|)\right)\mu\left(\frac{|y|}{R}+\varphi\right)u(y,t).
\end{eqnarray*}

Using the bounds on the cut-off functions and the fact that $|i\partial_t u+\lll u|=|Vu|\leq\norm{V}_\infty |u|$ we get
\begin{eqnarray*}
\abs{(i\partial_t+\mathcal{L})g}&\leq&
\norm{V}_\infty |u|+C_\ffi\abs{\mu'\left(\frac{|x|}{R}+\varphi\right)}|u|\\
&&+\frac{1}{q+1}\abs{\sum_{y\sim x}\left(\mu\left(\frac{|y|}{R}+\varphi\right)-\mu\left(\frac{|x|}{R}+\varphi\right)\right)u(y,t)}\\
&&+\frac{1}{q+1}\sum_{y\sim x}\abs{\theta^R(|y|)-\theta^R(|x|)}\abs{u(y,t)}.
\end{eqnarray*}

Thus, by means of the Carleman estimate with $\beta=1/\eps$ and $R$ large enough,
\begin{multline}\label{evolg}
\sinh(2\alpha/R^2)\cosh(4\alpha/\epsilon R)\|e^{\alpha\left(\frac{|x|}{R}+\varphi\right)^2}g\mathbf{1}_{|x|\geq1}\|_{L^2_{x,t}}\\
 \le
c\norm{V}_\infty^2\|e^{\alpha\left(\frac{|x|}{R}+\varphi\right)^2}g\|_{L^2_{x,t}}^2\\
+c\int_0^1\sum_{n\ge0,|x|=n}e^{2\alpha\left(\frac{n}{R}+\varphi\right)^2}\left|\mu'\left(\frac{n}{R}+\varphi\right)\right|^2|u(x,t)|^2\mbox{d}t\\
+c\int_0^1\sum_{n\ge0,|x|=n}\sum_{y\sim x}e^{2\alpha\left(\frac{n}{R}+\varphi\right)^2}\left|\mu\left(\frac{|y|}{R}+\varphi\right)-\mu\left(\frac{|x|}{R}+\varphi\right)\right|^2|u(y,t)|^2\mbox{d}t\\
+c\int_0^1\sum_{n\ge0,|x|=n}\sum_{y\sim x}e^{2\alpha\left(\frac{n}{R}+\varphi\right)^2}
\abs{\theta^R(|y|)-\theta^R(|x|)}\abs{u(y,t)}^2\mbox{d}t\\
+\int_0^1\sinh\frac{4\alpha}{R}\left(\frac{1/2}{R}+\varphi\right)\sum_{|x|=1}e^{2\alpha\left(\frac{1}{R}+\varphi\right)^2}|g(x,t)|^2\mbox{d}t.
\end{multline}
Note that we used Cauchy-Schwarz in the third and fourth sums in the form $\dst\sum_{x\in\tree}\abs{\sum_{y\sim x}\psi(y)}^2
\leq (q+1)\sum_{x\in\tree}\sum_{y\sim x}\abs{\psi(y)}^2$.

We now study carefully the support of each term.

For the first term involving $V$: by taking $\alpha=cR\log R$ with $c\geq\eps/2$
\begin{equation}\label{beha}
\sinh(2\alpha/R^2)\cosh(4\alpha/\epsilon R)\geq 
2cR^{\frac{4c}{\epsilon}-1}\log R,
\end{equation}
so that, when $R$ large enough (depending on $L$ also now),
the term  on the right, up to the term involving root $o$, is absorbed in the left-hand side. Further, the remaining term is bounded by $ce^{2\alpha(2+\epsilon^{-1})}L^2A^2$.

For the term involving the derivative of the function $\mu$, we easily see that $\frac{n}{R}+\varphi\le 1+\epsilon^{-1}$, and, therefore
$$
\int_0^1\sum_{n\ge0,|x|=n}e^{2\alpha\left(\frac{n}{R}+\varphi\right)^2}\left|\mu'\left(\frac{n}{R}+\varphi\right)\right|^2|u(x,t)|^2
\,\mbox{d}t
\le ce^{2\alpha\left(1+\epsilon^{-1}\right)}A^2.
$$

Next we study the term involving the difference of $\mu$ functions, which is similar to the last one. It is easy to check that if $\frac{n}{R}+\varphi\ge \epsilon^{-1}+1+\frac{1}{R}$ both functions $\mu$, the one evaluated at $x$ and the one evaluated at one neighbor of $x$ are 0. Hence,
\begin{multline*}
\int_0^1\sum_{n\ge0,|x|=n}\sum_{y\sim x}e^{2\alpha\left(\frac{n}{R}+\varphi\right)^2}
\left|\mu\left(\frac{|y|}{R}+\varphi\right)-\mu\left(\frac{|x|}{R}+\varphi\right)\right|^2|u(y,t)|^2\,\mbox{d}t\\
\le e^{2\alpha\left(\epsilon^{-1}+1+1/R\right)^2}A^2.
\end{multline*}

Now we focus on the term with difference of $\theta$ functions. In this case, the only possibilities where the difference is not zero are summarize as

--- $|x|=\lfloor R\rfloor -1$ and $y$ a future neighbor, $|y|=\lfloor R\rfloor$.

--- $|x|=\lfloor R\rfloor$ and $y$ any neighbor of $x$.

--- $|x|=\lfloor R\rfloor +1$ and $y$ the past neighbor, $|y|=\lfloor R\rfloor $.\\

Thus,
\begin{multline*}
\int_0^1
\sum_{n\ge0,|x|=n}\sum_{y\sim x}e^{2\alpha\left(\frac{n}{R}+\varphi\right)^2}
\abs{\theta^R(|y|)-\theta^R(|x|)}\abs{u(y,t)}^2\,\mbox{d}t\\
\le c e^{2\alpha\left(3+\epsilon^{-1}+1/R\right)^2}\lambda^2(R).
\end{multline*}

For the last term in the right-hand side, we just bound the function $\varphi$ to put all the functions out of the sum. Now, by the definition of $\theta^R$ and $\mu$, we see that if $x=x_0$ and $t\in[1/2-1/8,1/2+1/8]$ then $\left|\frac{|x_0|}{R}+\varphi e_1\right|=2+\epsilon^{-1}+2/R,$ so the cut-off functions are 1 and $g(x_0,t)=u(x_0,t)$. This allows us to bound the left-hand side of the Carleman inequality of the lemma by
\[
\|e^{\alpha\left(\frac{|x|}{R}+\varphi \right)^2}g\mathbf{1}_{|x|\geq1}\|_{L^2_{x,t}}^2\ge e^{(2+\epsilon^{-1}+2/R)^22\alpha},
\]
since $\int_{1/2-1/8}^{1/2+1/8}|u(x_0,t)|^2\ge 1$. 

Gathering all these results we have,
\begin{multline*}
\sinh\left(\frac{2\alpha}{R^2}\right)\cosh\left(\frac{4\alpha}{\epsilon R}\right)e^{2\alpha(2+\epsilon^{-1}+2/R)^2}\\
\\\le e^{2\alpha(2+\epsilon^{-1})}A^2L^2+\quad e^{2\alpha(1+\epsilon^{-1}+1/R))^2}A^2
\\+\quad\sinh\frac{4\alpha}{R}\left(\frac{1/2}{R}+2+\epsilon^{-1}\right)e^{2\alpha(2+\epsilon^{-1}+1/R)^2}A^2\\
+\quad e^{2\alpha(3+\epsilon^{-1}+1/R)^2}\lambda^2(R).
\end{multline*}

It is clear that the first two terms in the right-hand side are smaller than the third term. Let us see that the third term can be absorbed in the left-hand side, for $R$ large enough, depending on $A$ (recall that before we showed that $R$ depends on $L$ as well) and $\epsilon$, which is a fixed number. Indeed, taking into account that $\alpha=cR\log R$ with $c>\frac{\epsilon}{2}$, we have
\begin{multline*}
\sinh\left(\frac{2\alpha}{R^2}\right)\cosh\left(\frac{4\alpha}{\epsilon R}\right)e^{2\alpha(2+\epsilon^{-1}+2/R)^2}\\
\sim 2c\log RR^{2cR(2+\epsilon^{-1})^2+8c(2+\epsilon^{-1})+4c\epsilon^{-1}-1+8c/R}
\end{multline*}
and
\begin{multline*}
\sinh\frac{4\alpha}{R}\left(\frac{1/2}{R}+2+\epsilon^{-1}\right)e^{2\alpha(2+\epsilon^{-1}+1/R)^2}A^2\\
\sim A^2R^{2cR(2+\epsilon^{-1})^2+8c(2+\epsilon^{-1})+4c/R},
\end{multline*}
which proves our claim.

Finally, we conclude that
\[
1\le 2c\log RR^{\frac{4c}{\epsilon}-1}\le c_{\epsilon}e^{(5+2\epsilon^{-1})2c R \log R-(2+2\epsilon^{-1})2c\log R}\lambda^2(R),
\]
so
$$
\lambda(R)\ge c_{\epsilon}e^{-(5+2\epsilon^{-1})c R \log R+(2+2\epsilon^{-1})c\log R}.
$$

We just finish this result by taking $c=\epsilon/2+\epsilon^2$, to have
\begin{equation}\label{lambdaR}
\lambda(R)\ge c_{\epsilon}e^{-(1+9\epsilon/2+5\epsilon^2) R \log R+(1+3\epsilon+2\epsilon^2)\log R}
\end{equation}
which is of the desired form.
\end{proof}

Once we have the lower bound, since the previous log-convexity properties, {\it i.e.} Proposition \ref{prop:4.3}, derive upper bounds for the term $\lambda(R)$, we are in position to prove Theorem B from the introduction, that is

\begin{theorem}[Uniqueness result]\ \\
Let $u\in C^1([0,1]:\ell^2(\tree))$ be a solution of \eqref{schr} with $V$ a bounded potential. If for $\mu>1$
\[
\sum_{x\in\tree}e^{2\mu|x|\log(|x|+1)}\big(|u(x,0)|^2+|u(x,1)|^2\big)<+\infty,
\]
then $u\equiv0$.
\end{theorem}

\begin{proof} Let $\eta>0$ be such that $\mu>1+\eta>1$.

If $u$ is not zero, after eventually changing the root of the tree and multiplying $u$ by a constant, we may assume that 
there is an $x_0\in\tree$ with $|x_0|=2$ such that $u$ satisfies
\[
\int_{1/2-1/8}^{1/2+1/8}|u(x_0,t)|^2\,dt\ge1.
\]
We can then apply the previous theorem to obtain a lower bound for $\lambda(R)$. More precisely, we know that $\lambda(R)$ satisfies $\eqref{lambdaR}$. On the other hand, by Proposition \ref{prop:4.3} we have
\[
\sup_{t\in[0,1]}\sum_{x\in{\tree}}|u(x,t)|^2e^{2\mu |x|\log|x|}<+\infty.
\]

Hence $\lambda(R)\le ce^{-\mu R\log R }$. Combining both bounds,
$$
ce^{-\mu R\log R }\ge \lambda(R)\ge ce^{-(1+\eta) R \log R}.
$$
We get a contradiction letting $R\to\infty$.
\end{proof}


\section{Other infinite graphs}
\label{other}

Let $G=(\mathcal{E},\mathcal{V})$ be an infinite graph and $\mathcal{L}$ be the associated combinatorial Laplacian.
Assume that $G$ is such that $\mathcal{L}$ has a finitely supported eigenfunction $e_\lambda$\,:
$\mathcal{L}e_\lambda=\lambda e_\lambda$. In this case, the solution of
$$
i\partial_t u(x,t)=\mathcal{L}u(x,t),\ u(x,0)=e_\lambda(x)
$$
is given by $u(x,t)=e^{-i\lambda t}e_\lambda(x)$.  This solution is thus finitely supported at all times. 
In particular, no analogue of Theorems A and C can hold.

Examples of graphs where this may happen are the Diestel-Leader graphs introduced in \cite{DL}. Recall that those are defined as follows:

\begin{definition}
Let $q,r\geq 2$. In $\tree_q$ (resp $\tree_r$) we fix a geodesic ray $\omega$ (resp $\omega'$)
and write $h=h_\omega$ (resp. $h=h_{\omega'}$) for the associated Busemann function.
The \emph{Diestel-Leader graph} $DL(p,q)$ is
$$
DL(q,r)=\{(x,y)\in\tree_q\times\tree_r\,: h(x)+h(y)=0\}
$$
and neighbourhood is given by $(x,y)\sim(x',y')$ if $x\sim x'$ and $y\sim y'$.
\end{definition}

This graph is regular of degree $q+r$. Bartholdi and Woess \cite[Theorem 3.15]{BW} have shown that $L^2\bigl(DL(q,r)\bigr)$
has an orthonormal basis of finitely supported eigenfunctions of $\mathcal{L}$.

Quint \cite{Qu} and Taplyaev \cite{Te} have respectively shown that on the Pascal graph and the Sierpi\'nski graphs there
also exists finitely supported eigenfunctions of $\mathcal{L}$. Of course, on trees, there are no non-zero finitely supported 
eigenfunctions of the laplacian.

\smallskip

Let us now turn to non-homogeneous trees and prove the following:

\begin{proposition}
Let $(\omega_n)$ be a sequence of positive real numbers with $\omega_n\to 0$. Then there exists a rooted tree $\tree$ such that $\mathcal{L}$
has an eigenvector with $|e(x)|\leq C\omega_{|x|}$ for some $C>0$. In particular, if $u$ is the solution of 
$$
i\partial_t u(x,t)=\mathcal{L}u(x,t),\ u(x,0)=e(x)
$$
then $|u(x,t)|\leq C\omega_{|x|}$ for every $t\geq 0$.
\end{proposition}

Here $|x|$ means of course the distance to the root of $\tree$.

\begin{remark}
This does not mean that if $\tilde \omega_n=o(\omega_n)$ and $u$ is a solution of $i\partial u(x,t)=\mathcal{L}u(x,t)$
such that $|u(x,t_i)|=O(\tilde \omega_n)$ at times $t_0=0$ and $t_1=1$ then $u=0$.

For instance, for the homogeneous tree $\tree_q$, the construction below provides us with an eigenvector
for which $\omega_n=q^{-n/2}$. On the other hand, Corollary B shows that decrease rate
at which the only solution is $0$ is $\tilde \omega_n=e^{-\mu n\log(n+1)}$, $\mu>1$.

For the trees constructed below, we do not know what the maximal rate of decrease is, if such a maximal rate exists.
\end{remark}

\begin{proof} The tree we consider is a rapidly branching tree as introduced by Fujiwara \cite{Fu}.
Let $d_n$ be a sequence of integers with $d_n\geq 2$
and construct the tree recursively. We start with the root $o$.
We link $o$ to $d_0$ vertices. Each of these vertices is then linked to $d_1-1$ further vertices,... We thus construct a tree
such the vertices at distance $n$ from the root have degree $d_n$.

Next, we look for a radial eigenvector $e$ of $\mathcal{L}$ with eigenvalue $1$. For simplicity of notation, we
write $e(x)=e(|x|)$.
Those are constructed in \cite{Fu} but for sake of completeness, we reproduce the construction here.
Then, $\mathcal{L}e(x)=e(x)$ reads

--- if $x=0$, $e(0)-e(1)=e(0)$ thus $e(1)=0$

--- if $|x|=n\geq 1$, $\dst e(n)-\frac{1}{d_n}\bigl(e(n-1)+(d_{n}-1)e(n+1)\bigr)=e(n)$ thus
$e(n+1)=\dst-\frac{1}{d_n-1}e(n-1)$.

It follows that $e(n)=0$ if $n$ is odd and, if $n=2p\geq 2$,
$$
e(2p)=(-1)^p\left(\prod_{k=1}^p\frac{1}{d_{2k-1}-1}\right)e(0).
$$
It is then easy to inductively construct the $d_{2k-1}$'s in order to have $\prod_{k=1}^p(d_{2k-1}-1)\geq \omega_{2n}^{-1}$
and the corresponding $e$ is the eigenvector we are looking for.
\end{proof}

\section*{Acknowledgments}
The first author kindly acknowledge financial support from the IdEx postdoctoral program via the PDEUC project and from ERCEA Advanced Grant 2014 669689 - HADE.

The second author kindly acknowledge financial support from the French ANR program, ANR-12-BS01-0001 (Aventures), 
the French-Tunisian CMCU/UTIQUE project 32701UB Popart.

Both authors acknowledge the financial support of the Austrian-French AMADEUS project 35598VB - ChargeDisq.

This study has been carried out with financial support from the French State, managed
by the French National Research Agency (ANR) in the frame of the Investments for
the Future Program IdEx Bordeaux - CPU (ANR-10-IDEX-03-02).  

The authors wish to thank the anounymous referees for there helpfull comments.

\end{document}